\documentclass[12pt,twoside]{amsart}
\usepackage{amsfonts}
\usepackage{epsfig,graphics}
\usepackage{amssymb}
\usepackage{amscd}

\usepackage[numeric, abbrev, nobysame]{amsrefs}

\newtheorem{theorem}{Theorem}[section]
\newtheorem{proposition}[theorem]{Proposition}

\newtheorem{lemma}[theorem]{Lemma}

\newcommand{\ad}{\text{ad}}

\newcommand{\tr}{\text{tr}}

\newcommand{\id}{\text{Id}}

\newcommand{\BR}{{\bf R}}

\newcommand{\lieg}{{\mathfrak{g}}}

\newcommand{\liep}{{\mathfrak{p}}}

\newenvironment{Pf}{\medskip \noindent {\bf Proof: }}
   {$\diamondsuit$ }

\let\Cal\mathcal 
\let\Bbb\mathbb
\let\frak\mathfrak
\let\phi\varphi

\newcommand{\al}{\alpha}
\newcommand{\ep}{\epsilon}
\newcommand{\om}{\omega}

\newcommand{\ph}{\varphi}

\newcommand{\La}{\Lambda}
\newcommand{\Om}{\Omega}
\newcommand{\x}{\times}

\begin{document}
\pagenumbering{arabic}
\title[Essential Killing fields]{Essential Killing fields\\ 
of parabolic geometries:\\
  projective and conformal structures} \author{Andreas \v
  Cap}\address{Faculty of Mathematics\\University of
  Vienna\\Nordbergstr. 15\\1090 Vienna,
  AUSTRIA}\email{Andreas.Cap@univie.ac.at} 
\author{Karin Melnick}\address{Department of Mathematics\\
University of Maryland\\College Park, MD 20742, USA}
\email{karin@math.umd.edu}

\begin{abstract}
  We use the general theory developed in our article \cite{main} in
  the setting of parabolic geometries to reprove known results on
  special infinitesimal automorphisms of projective and conformal
  geometries. 
\end{abstract}

\subjclass{53A20, 53A30, 53B15, 58J70}
\keywords{essential infinitesimal automorphism, projective structure,
conformal structure, higher order fixed point}

\maketitle

\section{Introduction}

This text is a complement to our article \cite{main} which studies
special infinitesimal automorphisms in the general setting of parabolic
geometries. We illustrate the general theory developed there by
reproving the known results on such automorphisms in the cases of
projective structures (see \cite{Nagano-Ochiai}) and of conformal
structures (see \cite{Frances-Melnick} and \cite{Frances}). The main
moral of \cite{main} is that special infinitesimal automorphisms can be
understood to a large extent by doing purely algebraic computations on
the level of the Lie algebra which governs the geometry in
question. The output of these computations can then be nicely
interpreted geometrically, providing examples of the powerful
interplay between algebra and geometry for the class of parabolic
geometries. Even for well known geometries as the two examples
discussed here, this leads to precise new descriptions of the behavior of these special flows.

The two examples of structures discussed in this article belong to the
subclass of parabolic geometries related to so--called
$|1|$--gradings. These geometries have been studied under the names
AHS--structures, irreducible parabolic geometries, and abelian
parabolic geometries, in the literature. We will restrict all
discussions in this article to geometries in this subclass, referring
to \cite{main} for more general concepts. 

The notion of higher order fixed point, which is the
main concept studied in \cite{main}, becomes very simple in the
AHS--case: A point $x_0$ is a \emph{higher order fixed point} of an
infinitesimal automorphism $\eta$ of an AHS--structure if and only if
the local flow $\ph^t$ of $\eta$ fixes $x_0$ to first order, meaning $\ph^t(x_0)=x_0$ and $D\ph^t(x_0)=\id$ for all $t$ (see \cite[Def. 1.5]{main} for the general definition). Infinitesimal
automorphisms with this property exist on the homogeneous model of
each AHS--structure (see \ref{2.2} below) and, as observed in
\cite[Rmk. 1.6]{main}, they are always \textit{essential}. For the structures
treated in this article, the latter condition simply means not
preserving any affine connection in the projective class, respectively,
any metric in the conformal class.

The basic question is to what extent infinitesimal automorphisms
admitting a higher order fixed point can exist on non--flat
geometries. The simplest possible answer to this question would be that
existence of a higher order fixed point implies local flatness on an
open neighborhood of this fixed point, which happens for projective
structures. In this case, the infinitesimal automorphism in question
is conjugate via a local isomorphism to an infinitesimal automorphism
on the homogeneous model of the geometry which has a higher order
fixed point, and the latter can be explicitly described.

In general, and already in the case of conformal pseudo-Riemannian
structures, the situation is less simple. The next best case is that
existence of a higher order fixed point $x_0$ implies local flatness
on an open subset $U$ such that $x_0 \in \overline{U}$. For example,
for higher order fixed points of conformal flows of timelike or
spacelike type (see \ref{sect:conformal} below), the set $U$ is the
interior of the light cone, or the interior of its complement,
respectively, intersected with a neighborhood of $x_0$.  It seems
difficult to precisely describe the flow outside of $U$, but one can
obtain detailed information about the flow on $\overline{U}$ from the
results of \cite{main} (or from the results of \cite{Nagano-Ochiai}
and \cite{Frances-Melnick} in the projective and conformal cases,
respectively).

% While
% one can compare directly to the homogeneous model on $U$, there is no
% straightforward relation to the homogeneous model locally around
% $x_0$. Hence in this case it is highly desirable to get a more
% detailed description of the infinitesimal automorphism respectively of
% its flow and such a description is provided by the results in
% \cite{main}.

\section{Background}
In this section, we very briefly review some background on
AHS--structures referring to sections 3.1, 3.2, and 4.1 of
\cite{book}, and we collect the results from \cite{main} we will need
in the sequel.

\subsection{AHS--structures}\label{2.1} 
The basic data needed to specify an AHS--structure is a semi\-simple
Lie algebra $\lieg$ endowed with a $|1|$--grading, a
decomposition $\lieg=\lieg_{-1}\oplus\lieg_0\oplus\lieg_1$ making
$\lieg$ into a graded Lie algebra. This means that $\lieg_{\pm 1}$ are
abelian Lie subalgebras, while $\lieg_0$ is a Lie subalgebra which
acts on $\lieg_{\pm 1}$ via the restriction of the adjoint action. The
only additional information encoded in $\lieg$ is the restriction of
the Lie bracket of $\lieg$ to a map $[\ ,\
]:\lieg_1\otimes\lieg_{-1}\to\lieg_0$. It turns out that $\lieg_0$ is
always reductive, and its action on $\lieg_{-1}$ defines a faithful
representation. Finally $\frak p:=\lieg_0\oplus\lieg_1$ is a maximal
parabolic subalgebra of $\lieg$ and $\lieg_1$ is an ideal in $\frak
p$. The Killing form of $\lieg$ induces an isomorphism
$(\lieg/\liep)^*\cong\lieg_1$ of $\frak p$--modules, which can also be
interpreted as an isomorphism $(\lieg_{-1})^*\cong\lieg_1$ of
$\lieg_0$--modules.

Choosing a Lie group $G$ with Lie algebra $\lieg$, it turns out that
the normalizer $N_G(\frak p)$ is a closed subgroup of $G$ with Lie
algebra $\frak p$. One next chooses a \textit{parabolic subgroup}
$P\subset G$ corresponding to $\frak p$, a subgroup lying between
$N_G(\frak p)$ and its connected component of the identity. Then one
defines $G_0\subset P$ to be the closed subgroup consisting of all
elements whose adjoint action preserves the grading of $\lieg$. This
subgroup has Lie algebra $\lieg_0$ and it naturally acts on $\lieg_{-1}$ via
the adjoint action. Finally, the exponential mapping restricts to a
diffeomorphism from $\lieg_1$ onto a closed normal subgroup
$P_+\subset P$ and $P$ is the semidirect product $G_0\ltimes P_+$.

An AHS--structure of type $(\lieg,P)$ on a smooth manifold $M$ which
has the same dimension as $G/P$ is then defined as a \emph{Cartan
  geometry} of that type, which consists of a principal $P$--bundle
$p:B\to M$ and a \textit{Cartan connection} $\om\in\Om^1(B,\lieg)$.
The Cartan connection trivializes the tangent bundle $TB$, is
equivariant with respect to the principal right action of $P$, and
reproduces the generators of fundamental vector fields (see
\cite[Sec. 1.5]{book} for the definition). The \textit{homogeneous
  model} of the geometry is the bundle $G\to G/P$ with the
left-invariant Maurer--Cartan form as the Cartan connection.

Via the Cartan connection, $TM$ can be identified with
the associated bundle $B\x_P(\lieg/\liep)$. More precisely, one can
form $B_0:=B/P_+$ which is a principal bundle over $M$ with structure
group $P/P_+\cong G_0$, and $\om$ descends to a soldering form on this
bundle. Hence one can interpret $B_0\to M$ as a first order structure
with structure group $G_0$. In this interpretation, $TM\cong
B_0\x_{G_0}\lieg_{-1}$, and consequently $T^*M\cong
B_0\x_{G_0}\lieg_1\cong B\x_P\lieg_1$. 

If none of the simple ideals of $\frak g$ is of the type corresponding
to projective structures, requiring a normalization condition on the
curvature of $\om$ makes the Cartan geometry $(B,\om)$ equivalent, in
the categorical sense, to the underlying first order structure
$B_0$. In the projective case, this underlying structure contains no
information and the Cartan geometry is equivalent to the choice of a
projective class of torsion free linear connections on the tangent
bundle $TM$.

For any Cartan geometry, the curvature of a Cartan connection is a
complete obstruction to local isomorphism to the homogeneous model, or
\emph{local flatness} (see \cite[Prop. 1.5.2]{book}). In the
case of AHS--structures, there is a conceptual way to extract parts of
this curvature, called \emph{harmonic curvature} components, which
still form a complete obstruction to local flatness (see \cite[Thm.
3.1.12]{book}). In the examples we are going to discuss, there is only
one harmonic curvature component, and this is either a version of Weyl
curvature or of a Cotton--York tensor.

\subsection{Higher order fixed
  points}\label{2.2}
 
An infinitesimal automorphism of a Cartan geometry $(B\to M,\om)$ of
type $(\lieg,P)$ is a vector field $\tilde\eta\in\frak X(B)$ which is
$P$--invariant and satisfies $\Cal L_{\tilde\eta}\om=0$, where $\Cal
L$ denotes the Lie derivative. Any $P$--invariant vector field
$\tilde\eta$ on $B$ is projectable to a vector field $\eta\in\frak
X(M)$. From the equivalence to underlying structures discussed in
the previous section, one concludes that, apart from the projective case,
$\tilde\eta$ is an infinitesimal automorphism of the Cartan geometry
if and only if $\eta$ is an infinitesimal automorphism of the
underlying first order structure. There is an analogous correspondence
in the projective case.

Via the Cartan connection $\om$, a $P$--invariant vector field
$\tilde\eta$ on $B$ corresponds to a $P$--equivariant function
$f:B\to\lieg$. Then $\tilde\eta$, or the underlying vector
field $\eta$ on $M$, has a higher order fixed point at
$x_0\in M$ if and only if for one---or equivalently, any---point $b_0\in
B$ with $p(b_0)=x_0$, we have $f(b_0)\in\lieg_1\subset\lieg$. (Observe that $x_0$ is fixed by the flow, or is a zero of $\eta$, if and only if $f(b_0)\in\liep\subset\lieg$.)

If $\tilde\eta$ has a higher order fixed point at $x_0$, then for each
$b_0\in B$ with $p(b_0)=x_0$ the value $f(b_0)\in\lieg_1$ via $b_0$
corresponds to an element of $T^*_{x_0}M$ since $T^*M\cong
B\x_P\lieg_1$. Equivariance of $f$ implies that this element does not
depend on $b_0$, so $\tilde\eta$ gives rise to a well defined element
$\al\in T^*_{x_0}M$, called the \textit{isotropy} of $\tilde\eta$ (or
of $\eta$) at the higher order fixed point $x_0$ \cite[Def. 1.5]{main}.

The $P$--orbit $f(p^{-1}(x_0))\subset\lieg_1$
similarly depends only on $\al$. This $P$--orbit is called the \textit{geometric
  type} of the isotropy $\al$ \cite[Sec. 1.2]{main}, and is the most basic invariant of a
higher order fixed point. In all examples studied here and in
\cite{main}, different geometric types are discussed separately.

It is possible to do precise calculations for infinitesimal
automorphisms with higher order fixed points on the homogeneous model
$G/P$ of the geometry. Here the automorphisms are exactly the left
translations by elements of $G$, so infinitesimal automorphisms are
right invariant vector fields on $G$. By homogeneity, we may assume
that our automorphism fixes the point $o=eP\in G/P$. Thus the value of
an infinitesimal automorphism at $o$ can be naturally interpreted as
the generator $Z\in \lieg$ of the right invariant vector field in
question. Now of course $o$ is a fixed point if and only if $Z\in\frak
p$ and a higher order fixed point if and only if $Z\in\lieg_1$. Hence
the flows of such infinitesimal automorphisms are just the left
translations by $e^{tZ}$ for $Z\in\lieg_1$.

Given a higher order fixed point $x_0$ of an infinitesimal
automorphism $\eta$ and a neighborhood $U$ of $x_0$ in $M$, we call
the \textit{strongly fixed component} of $x_0$ in $U$ \cite[Def. 2.3]{main} the set of all
points in $U$ which can be reached from $x_0$ by a smooth curve lying
in $U$, all of whose points are higher order fixed points of the same
geometric type as $x_0$. The higher order fixed point $x_0$ is called
\textit{smoothly isolated} if its strongly fixed component consists of
$\{x_0\}$ only.

\subsection{Results from \cite{main}}\label{2.3} 
The first step to apply the results of \cite{main} is to associate to
an element $\al\in T^*_{x_0}M$ three geometrically significant subsets of the tangent space
$T_{x_0}M$. These sets are defined algebraically, after choosing an element
$b_0\in B$ with $p(b_0)=x_0$. Via $b_0$, the covector $\al$
corresponds to an element $Z\in\lieg_1$, and the subsets of
$\lieg_{-1}$ are defined by
\begin{gather*}
  C(Z):=\{X:[X,Z]=0\}\subset\{X:[X,[X,Z]]=0\}=:F(Z)\\
  T(Z):=\{X:[[Z,X],X]=-2X,[[Z,X],Z]=2Z\}
\end{gather*}
In \cite{main} the sets were denoted by $C_{\lieg_-}(Z)$, and similarly
for $F$ and $T$. Moreover, the original definition of $F_{\lieg_-}(Z)$
used there is different, but equivalence to the one used here is
observed in the proof of Proposition 2.16 of \cite{main}. 

For $X\in T(Z)$, we can consider $A:=[Z,X]\in\lieg_0$. Given a
representation $\Bbb W$ of $G_0$ on which $A$ acts diagonalizably, we
define $\Bbb W_{ss}(A)\subset\Bbb W_{st}(A)\subset\Bbb W$ as the sum
of all eigenspaces corresponding to negative eigenvalues, respectively,
to non--positive eigenvalues of $A$.

Identifying $\lieg_{-1} \cong \lieg/\liep$ as vector spaces, one can
use the element $b_0$ to identify $C(Z)\subset F(Z)$ and $T(Z)$ with
subsets $C(\al)\subset F(\al)$ and $T(\al)$ of
$T_{x_0}M$. Equivariance of $\om$ then easily implies that the latter
subsets are independent of the choice of $b_0$, so they are
intrinsically associated to $\al\in T_{x_0}^*M$. Observe that $C(\al)$
is a linear subspace of $T_{x_0}M$, $F(\al)$ is only closed under
multiplication by scalars, while $T(\al)$ is just a subset.

The most concise way to formulate the results of \cite{main} for our
purposes is via normal coordinates centered at a point $x_0\in M$, see
\cite[Sec. 1.1.2]{main}. For $X\in\lieg_{-1}$ we can take the
``constant vector field'' $\tilde X\in\frak X(B)$ characterized by
$\om(\tilde X)=X$. Write $\varphi_{\tilde X}^t(b_0)$ for the flow line
of such a field emanating from $b_0\in B$ with $p(b_0)=x_0$; then we can use
$X\mapsto p(\varphi_{\tilde X}^1(b_0))$ to define a diffeomorphism from an
open neighborhood of $0$ in $\lieg_{-1}$ onto an open neighborhood of
$x_0$ in $M$. Combining the inverse of such a diffeomorphism with the
identification of $\lieg_{-1}$ with $T_{x_0}M$ provided by $b_0$, we
obtain a \textit{normal coordinate chart} centered at $x_0$ with
values in $T_{x_0}M$. Varying $b_0$ gives a family of charts
parametrized by the fiber $p^{-1}(x_0)\cong P$.

Now the first result we need concerns the form of the flow, and in
particular further fixed points of an infinitesimal
automorphism admitting one higher order fixed point. It is proved in
Propositions 2.5 and 2.17 of \cite{main}.
\begin{proposition}\label{prop:form}
  Let $\eta$ be an infinitesimal automorphism of an AHS--structure
  having a higher order fixed point at $x_0\in M$ with isotropy
  $\al\in T^*_{x_0}M$.
 
(1) If $C(\al)=0$, then the higher order fixed point $x_0$ is smoothly
isolated. 

(2) In any normal coordinate chart centered at $x_0$ with values in a
neighborhood $U\subset T_{x_0}M$ of zero we have
\begin{itemize}
\item[(a)] Any point of $U\cap F(\al)$ is a zero of $\eta$. 
\item[(b)] Any point of $U\cap C(\al)$ is a higher order fixed point
  of the same geometric type as $x_0$.
\item[(c)] For any $\xi\in T(\al)$ the flow on the intersection
  $U\cap\BR\cdot\xi$ is given by $\ph^t(s\xi)=\tfrac{s}{1+st}\xi$ for
  $ts>0$.
\end{itemize}
\end{proposition}

The second result concerns information on local flatness coming from
elements of $T(\al)$. 
\begin{proposition}\label{prop:flat}
  Let $\eta$ be an infinitesimal automorphism of an AHS--structure
  having a higher order fixed point at $x_0\in M$ with isotropy
  $\al\in T^*_{x_0}M$. Suppose that for some element $Z\in\lieg_1$
  belonging to the geometric type of $\al$ and all $X\in T(Z)$, the
  element $A=[Z,X]$ acts diagonalizably on $\lieg_{-1}$ and on each
  representation $\Bbb W$ of $G_0$ in which one of the harmonic
  curvature components of the geometry in question has its
  values. Suppose further that the following conditions are satisfied:
  \begin{enumerate}
  \item For each $X\in T(Z)$, all eigenvalues of $A$ on $\lieg_{-1}$ are
    non--positive and the $0$--eigenspace coincides with $C(Z)$.
  \item For each $X\in T(Z)$, $\Bbb W_{ss}(A)=0$.
  \item $\cap_{X\in T(Z)}\Bbb W_{st}(A)=0$. 
  \end{enumerate}

  Then in each normal coordinate chart centered at $x_0$ with values in
  a neighborhood $U\subset T_{x_0}M$ of zero and for each $\xi\in
  T(\al)$, there is an open neighborhood of
  $(\BR\cdot\xi)\cap(U\setminus 0)$ on which the geometry is locally
  flat. In particular, we get local flatness on an open subset
  containing $x_0$ in its closure.
\end{proposition}
\begin{proof}
  Using condition (3) we can apply Proposition 2.15 of \cite{main} to
  conclude that all harmonic curvature quantities vanish in the fixed
  point $x_0$. But this argument also applies to any other higher
  order fixed point of the same geometric type as $x_0$. Thus we may
  apply Corollary 2.14 of \cite{main} which directly gives the result.
\end{proof}

\section{Results}

 \subsection{Projective structures}\label{sect:projective}
A projective structure on a smooth manifold $M$ of dimension $n\geq 2$
is given by an equivalence class of torsion free linear connections on
$TM$ which share the same geodesics up to parametrization. An
infinitesimal automorphism in this case can be simply defined as a
vector field $\eta$ on $M$ whose flow preserves this class of
connections or equivalently the family of geodesic paths. What we
will prove here is
\begin{theorem}[compare \cite{Nagano-Ochiai}] 
\label{thm:proj}
  Let $(M,[\nabla])$ be a smooth manifold of dimension $n\geq 2$
  endowed with a projective structure and suppose that $\eta\in\frak
  X(M)$ is an infinitesimal automorphism of the projective structure
  which has a higher order fixed point at $x_0\in M$. Then there is an
  open neighborhood of $x_0\in M$ on which the projective structure is
  locally flat. 
\end{theorem}

To describe projective structures as AHS--structures, take
$G=PSL(n+1,\BR)$ and let $P\subset G$ be the stabilizer of a point for
the canonical action of $G$ on ${\bf RP}^n$. On the level of Lie algebras, $\frak g=\frak{sl}(n+1,\BR)$, and the $|1|$--grading satisfies
$\lieg_0\cong\frak{gl}(n,\BR)$, $\lieg_{-1}\cong \BR^n$, and
$\lieg_1\cong \BR^{n*}$. The grading comes from a block decomposition
with blocks of sizes $1$ and $n$ of the form
 $$
 \begin{pmatrix}  -\tr(A) & Z \\ X & A\end{pmatrix}
 $$
 with $X\in \BR^n$, $A\in\frak{gl}(n,\BR)$, and $Z\in\BR^{n*}$. 

 Note $G_0 \cong GL(n,\BR)$. It follows that there is just one
 non--zero $P$--orbit in $\lieg_1$, and hence just one possible
 geometric type of isotropy. The algebra needed for our purpose is in
 the following lemma:
\begin{lemma}\label{lem:proj-alg}
  For $0\neq Z\in\lieg_1$,
  \begin{gather*}
    0= C(Z)\subset F(Z)=\{X\in\lieg_{-1}:ZX=0\}
    \\
    T(Z)=\{X\in\lieg_{-1}:ZX=1\}
  \end{gather*}

  Moreover, $Z$ satisfies all conditions of Proposition \ref{prop:flat}.
 \end{lemma}
 \begin{Pf}
   For $Z\in\lieg_1$ and $X,Y\in\lieg_{-1}$, the brackets relevant for
   our purposes are 
\begin{eqnarray*}
[Z,X ] & = & -XZ \\
\left[ \left[ Z,X \right]   ,Y \right]  & =  & -ZYX-ZXY \\
\left[ \left[ Z,X \right] ,Z \right] & = & 2ZXZ.
\end{eqnarray*}
 The descriptions of $C(Z)$, $F(Z)$, and $T(Z)$ follow easily.

For $X\in T(Z)$, setting $A = [Z,X]$ gives $[A,Y]=-Y-ZYX$ for all
$Y\in\lieg_{-1}$. Hence the eigenspace decomposition of $\lieg_{-1}$
is $\BR\cdot X\oplus\ker(Z)$ with corresponding eigenvalues
$-2$ and $-1$. In particular, condition (1) of Proposition
\ref{prop:flat} holds.

Now it is well known that for projective structures, there always is
only one harmonic curvature component, see Section 4.1.5 of
\cite{book}. If $n>2$, the harmonic curvature is the so--called
\textit{Weyl curvature}, the totally tracefree part of the
curvature of any connection in the projective class. The corresponding
representation $\Bbb W$ of $\lieg_0$ is contained in
$\La^2\lieg_1\otimes\frak{sl}(\lieg_{-1})$.  From the calculations for $\lieg_-$, one can see that the
possible eigenvalues of $A$ on $\La^2\lieg_1$ are
$2$ and $3$, while the possible eigenvalues on $\frak{sl}(\lieg_{-1})$
are $-1$, $0$, and $1$. Hence $\Bbb W_{st}(A)=0$ for $n>2$. For $n=2$
the basic invariant is a projective analog of the Cotton--York tensor,
and the corresponding representation $\Bbb W$ is contained in
$\La^2\lieg_1\otimes\lieg_1$. Clearly, $A$ has only positive
eigenvalues on this representation, so $\Bbb W_{st}(A)=0$ holds for $n \geq 2$. The remaining conditions from Proposition \ref{prop:flat} follow.
 \end{Pf}

\begin{Pf}[of Theorem \ref{thm:proj}]
  The description of $T(Z)$ in the lemma shows that the set of non--zero
  scalar multiples of elements of $T(\al)\subset T_{x_0}M$ is the complement of the hyperplane $\ker(\al)$. In particular, the intersection of this set with the range $U$ of a normal
  coordinate chart is a dense subset of $U$. Now by
  Proposition \ref{prop:flat}, the geometry is locally flat on this
  dense subset and hence on all of $U$.
\end{Pf}

\subsection{Pseudo--Riemannian conformal
  structures}
  \label{sect:conformal} 

  Recall that two pseudo--Riemannian metrics $g$ and $\hat g$ are
  conformally equivalent if there is a positive smooth function $f$
  such that $\hat g=f^2g$. Evidently, conformally equivalent metrics
  have the same signature. A conformal equivalence class of metrics on
  a smooth manifold $M$ is called a \emph{conformal structure} on
  $M$. This can be equivalently described as a first order structure
  corresponding to $CO(p,q)\subset GL(p+q,\BR)$.  There are three
  orbits of $CO(p,q)$ on $\BR^{(p+q)*}$, when $0 < p \leq q$.  These
  orbits give three \emph{geometric types} (see \cite[Sec 1.2]{main}):
  spacelike, null, or timelike, and correspond to the sign of the
  inner product of the vector with itself.  For definite signature,
  there is just one possible geometric type. As we shall see, flows
  with isotropy equal to a spacelike or timelike element of
  $\BR^{p,q*}$ behave very similarly, so the main distinction is
  between null isotropy and non--null isotropy. Of course, in definite
  signature, only non--null isotropy is possible. Recall that
  infinitesimal automorphsism of a conformal structure are conformal
  Killing vector fields. We are going to prove the following.

\begin{theorem}[compare \cite{Frances}, \cite{Frances-Melnick}]
\label{thm.conformal}
Let $(M,[g])$ be a smooth manifold of dimension $\geq 3$ endowed with
a conformal structure, and let $\eta$ be a conformal Killing vector
field on $M$. Then higher order fixed points with non--null isotropy
are smoothly isolated, while for null isotropy there is a smooth curve
contained in the strongly fixed component. Moreover, for any higher
order fixed point $x_0$ of $\eta$, there is an open subset $U\subset
M$ with $x_0\in\overline{U}$ on which $M$ is locally conformally flat.
\end{theorem}

This theorem is a consequence of the detailed descriptions of
infinitesimal automorphisms with the two kinds of higher order fixed
points in Propositions \ref{prop:conf-non-null} and
\ref{prop:conf-null}, respectively.

The description of pseudo--Riemannian conformal structures in
dimension $n\geq 3$ as parabolic geometries is well known, see
Sections 1.6 and 4.1.2 of \cite{book}.  A structure of signature
$(p,q)$, is modeled on $G/P$ with $G=PO(p+1,q+1)$ and $P$ the
stabilizer of a null line in $\BR^{p+1,q+1}$. The homogeneous model is
the space of null lines in $\BR^{p+1,q+1}$, a quadric in ${\bf RP}^{n+1}$.  The Lie algebra $\frak g=\frak{so}(p+1,q+1)$ can be
realized as

$$
\lieg =
 \left \{
\begin{pmatrix} a & Z & 0 \\ X & A & -\Bbb I Z^t\\ 
0 & -X^t\Bbb I & -a
 \end{pmatrix}
: 
\begin{array}{c} 
a \in \BR, \\
X \in \BR^n,\ Z \in \BR^{n*},  \\
A\in\frak{so}(p,q)
\end{array}
\right \}
$$

Here $\Bbb I$ is the diagonal matrix $\id_p \oplus - \id_q$. The
grading corresponding to $P$ has the form
$\lieg=\lieg_{-1}\oplus\lieg_0\oplus\lieg_1$ with the components
represented by $X$, $(A,a)$ and $Z$, respectively. The relevant
bracket formulae for $Z\in\lieg_1$ and $X,Y\in\lieg_{-1}$, are:
\begin{equation*}
\begin{aligned}
{} [Z,X]&=(-XZ+\Bbb I(XZ)^t\Bbb I,ZX)\\
[[Z,X],Z]&= 2ZXZ-Z\Bbb IZ^tX^t\Bbb I= 2ZXZ-\langle Z,Z\rangle X^t\Bbb I\\
[[Z,X],Y]&=-XZY+\Bbb I Z^tX^t\Bbb I Y-ZXY\\
         &=-XZY-ZXY+\langle X,Y\rangle \Bbb I Z^t. 
\end{aligned}\tag{$*$}
\end{equation*}
 where $\langle\ ,\ \rangle$ denotes the (standard) inner product of
 signature $(p,q)$ on $\lieg_{\pm 1}$ corresponding to the matrix
 $\Bbb I$. This leads to $G_0\cong CO(p,q)$ and $P\cong CO(p,q)
 \ltimes \BR^{n*}$. 

 In all dimensions $n\neq 4$, there is just one harmonic curvature
 component. For $n\geq 5$, this component is the Weyl curvature, the
 totally tracefree part of the Riemann curvature tensor of any metric
 in the conformal class. It is a section of the bundle associated to
 the irreducible component of highest weight in
 $\La^2\lieg_1\otimes\frak{so}(\lieg_{-1})$. For $n=3$, the harmonic
 curvature is the Cotton--York tensor, which is a section of the
 bundle associated to the irreducible component of highest weight in
 $\La^2\lieg_1\otimes\lieg_1$. In dimension $n=4$, the Weyl curvature
 splits into two components according to the splitting of
 $\La^2\lieg_1$ into a self--dual and an anti--self--dual part (with
 respect to the Hodge--$*$--operator), and these comprise the harmonic
 curvature.

\subsubsection{Non--null isotropy}
This lemma collects the needed algebraic results. Recall that for a
$|1|$--graded Lie algebra $\lieg$ there is a unique element
$E\in\lieg$, called the \textit{grading element}, such that for
$i=-1,0,1$ the subspace $\lieg_i$ is the eigenspace with eigenvalue
$i$ for $\ad(E)$. 
  
\begin{lemma}\label{lem:conf-nonnull-alg}
  If $Z\in\lieg_1$ is such that $\langle Z,Z\rangle\neq 0$, then
  \begin{enumerate}
  \item  The sets associated to $Z$ are
  \begin{gather*} 0=C(Z)\subset
    F(Z)= \{ X\in\lieg_{-1}:ZX=\langle X,X\rangle=0 \}
    \\ T(Z)= \left\{\frac{2}{\langle Z,Z\rangle}\Bbb I Z^t \right\}
    \end{gather*}
  
  \item For the unique element $X\in T(Z)$, the bracket
    $A=[Z,X]\in\lieg_0$ is twice the grading element of $\lieg$. In
    particular, all conditions of Proposition \ref{prop:flat} are
    satisfied by $Z$.
  \end{enumerate}
\end{lemma}

\begin{Pf}
  Assume $[Z,X] = 0$.  Then from the brackets in equation ($*$) above,
  $ZX = 0$ and $[[Z,X],Z] = 2ZXZ - \langle Z, Z \rangle X^t \Bbb I =
  0$.  Since $\langle Z, Z \rangle \neq 0$ by assumption, we conclude
  $X = 0$, so $C(Z)=0$. Next, suppose that $[[Z,X],X]=0$. Then formula
  ($*$) gives $2ZX X = \langle X,X\rangle\Bbb{I} Z^t$, which is
  impossible if $X$ and $\Bbb I Z^t$ are linearly dependent, so the
  description of $F(Z)$ follows.

  Next, if $[[Z,X],Z]$ is a multiple of $Z$, then $X^t\Bbb I$ must be
  a multiple of $Z$.  It is easy to compute this multiple, which
  yields $T(Z)$, and to verify that $A=[Z,X]$ is twice the grading
  element. This implies that $A$ always acts diagonalizably on all
  representations of $\lieg_0$, and it acts by multiplication by $-2$
  on $\lieg_{-1}$.  Since all eigenvalues of the grading element on
  $\La^2\lieg_1\otimes\lieg$ are positive, we conclude that $\Bbb
  W_{st}(A)=0$ for any representation $\Bbb W$ corresponding to a
  harmonic curvature component. \end{Pf}

From the general theory of parabolic geometries it follows that, given
any normal coordinate chart centered at $x_0$, straight lines through
zero in $T_{x_0}M$ correspond to distinguished curves of the geometry
emanating from $x_0$ (see \cite[Sec 5.3]{book}).  In conformal
geometry, these curves are conformal circles in non--null directions
and null geodesics in null directions.  While the latter are
determined by their initial direction up to a projective family of
reparametrizations, there is additional freedom for conformal circles.

Now the results of Lemma \ref{lem:conf-nonnull-alg} can be converted to geometry.

\begin{proposition}\label{prop:conf-non-null}
  Let $(M^n,[g])$ be a smooth manifold endowed with a
  pseudo--Riemannian conformal structure of signature $(p,q)$, where
  $n=p+q\geq 3$. Suppose $\eta$ is a conformal Killing vector field
  with a higher order fixed point at $x_0\in M$ with non--null
  isotropy $\al\in T^*_{x_0}M$.

\begin{enumerate}
\item The higher order fixed point $x_0$ is smoothly isolated. Null
  ge\-o\-de\-sics emanating from $x_0$ in directions in $\ker(\al)$ are,
  locally around $x_0$, zeros of $\eta$.

\item Let $\xi_0\in T_{x_0}M$ be the unique vector such that
  $g_x(\xi_0,-)$ is a non--zero multiple of $\al$ and such that
  $\al(\xi_0)=2$. Then for any conformal circle $c=c(s)$ in $M$ with
  $c(0)=x_0$ and $c'(x_0)=\xi_0$, there is $\ep>0$ such that
  $\ph^t_{\eta}(c(s))=c(\frac{s}{1+st})$ for $|s| <\ep$ and $st > 0$.
  Finally, an open neighborhood of $\{c(s): 0< |s| <\ep\}$, which in
  particular contains $x_0=c(0)$ in its closure, is conformally flat.
  \end{enumerate}
\end{proposition}
\begin{Pf}
Varying the point $b_0 \in p^{-1}(x_0)$ gives all normal coordinate charts centered at $x_0$, which give all conformal
circles emanating from $x_0$.  Now the result
follows from Propositions \ref{prop:form} and \ref{prop:flat}, which
apply in our case by Lemma \ref{lem:conf-nonnull-alg}.
\end{Pf}

\subsubsection{Null isotropy} 
Again we start with an algebraic lemma:
\begin{lemma}\label{lem:conf-null-alg}
  Let $0\neq Z\in\lieg_1$ be such that $\langle Z,Z\rangle=0$. Then
  the sets associated to $Z$ are
  \begin{gather*}
    \BR\cdot\Bbb IZ^t=C(Z)\subset F(Z)=\{X\in\lieg_{-1}:ZX=\langle
    X,X\rangle=0\}\\ T(Z)=\{X\in\lieg_{-1}:ZX=1, \langle
    X,X\rangle=0\},
  \end{gather*}
and all conditions of Proposition \ref{prop:flat} are satisfied by $Z$. 
\end{lemma}

\begin{Pf}
  As in the proof of Lemma \ref{lem:conf-nonnull-alg}, $[Z,X]=0$
  implies $ZX=0$, and then $[[Z,X],Y]=0$ implies that $X$ and $\Bbb
  IZ^t$ must be linearly dependent. Conversely, $[Z,\Bbb IZ^t]=0$ is
  easily verified, so the claim on $C(Z)$ follows. If $X$ is linearly
  independent of $\Bbb IZ^t$, then $0 = [[Z,X],X] =-2XZX+\langle X, X
  \rangle \Bbb I Z^t$ implies $ZX=\langle X,X\rangle=0$ and thus the
  description of $F(Z)$.

  Since $\langle Z,Z\rangle=0$, the equation $[[Z,X],Z]=2Z$ is
  equivalent to $ZX=1$. In this case, $[[Z,X],X]=-2X$ is equivalent to
  $\langle X,X\rangle=0$, and the description of $T(Z)$
  follows. Taking $X\in T(Z)$, the eigenspace decomposition
  of $\lieg_{-1}$ with respect to $A=[Z,X]$ follows from the formula
  for $[[Z,X],Y]$: The isotropic lines spanned by $X$ and by $\Bbb
  IZ^t$ are the eigenspaces for eigenvalues $-2$ and $0$,
  respectively, and $A$ acts by multiplication by $-1$ on the
  complementary subspace $\ker(Z)\cap X^\perp$. In particular,
  condition (1) from Proposition \ref{prop:flat} is satisfied. 

  The grading on $\lieg_{-1}$ is the one on $\BR^{p+1,q+1}$ that gives
  rise to the initial $|1|$--grading on $\lieg$, shifted by one
  degree. This shift does not change the induced grading of
  $\frak{so}(\lieg_{-1})$, which thus has eigenvalues $-1$, $0$, and
  $1$. On the other hand, the eigenvalues on $\La^2\lieg_1$ are $1$,
  $2$, and $3$. Hence for any subrepresentation $\Bbb W$ of
  $\La^2\lieg_1\otimes\lieg_1$, we have $\Bbb W_{st}(A)=0$, while for
  a subrepresentation $\Bbb W$ of $\La^2\lieg_1\otimes
  \frak{so}(\lieg_{-1})$, we have $\Bbb W_{ss}(A)=0$, and $\Bbb
  W_{st}(A)$ is contained in the tensor product of the degree one part
  of $\La^2\lieg_1$ with the degree $-1$ part of
  $\frak{so}(\lieg_{-1})$. This implies that all values in
  $\frak{so}(\lieg_{-1})$ of any element of $\Bbb W_{st}(A)$ act
  trivially on the $(-2)$--eigenspace in $\lieg_{-1}$ and thus on
  $X$. But $\lieg_{-1}$ is spanned by elements of $T(Z)$, so
  $\cap_{X\in T(Z)}\Bbb W_{st}(A)=0$ follows.
\end{Pf}

Again, this is easily converted into geometric information.

\begin{proposition}\label{prop:conf-null}
  Let $(M^n,[g])$ be a conformal manifold of signature $(p,q)$, where
  $n=p+q\geq 3$. Suppose that $\eta$ is a conformal Killing field with
  a higher order fixed point at $x_0\in M$ with null isotropy $\al\in
  T^*_{x_0}M$.  Let $F(\al)\subset T_{x_0}M$ be the set of those null
  vectors $\xi$ which satisfy $\al(\xi)=0$, and let $C(\al)\subset
  F(\al)$ be the line of all elements dual to a multiple of $\al$.

  Then for any normal coordinate chart centered at $x_0$ with values
  in $T_{x_0}M$, there is an open neighborhood $U$ of zero in
  $T_{x_0}M$ such that:

\begin{enumerate}
\item Elements of $F(\al)\cap U$ are zeros of $\eta$ and $C(\al)\cap
  U$ is contained in the strongly fixed component of $x_0$.

\item For any null vector $\xi\in U$ with $\al(\xi)\geq 0$, we get
  $\ph^t_{\eta}(\xi)=\tfrac{1}{1+t\al(\xi)}\xi$.

\item Any null vector $\xi\in U$ with $\al(\xi)>0$ has an open
  neighborhood on which the geometry is flat. In particular, the
  geometry is flat on an open set containing $x_0$ in its closure.
  \end{enumerate}
\end{proposition}
\begin{proof}
  Taking into account that a vector $\xi$ in part (2) is either a
  multiple of an element of $T(\al)$ (if $\al(\xi)\neq 0$) or
  contained in $F(\al)$ (if $\al(\xi)=0$), the result follows directly
  from Propositions \ref{prop:form} and \ref{prop:flat}. 
\end{proof}


\begin{bibdiv}
\begin{biblist}

\bib{main}{article}{
   author={{\v{C}}ap, Andreas},
   author={Melnick, Karin},
   title={Essential Killing fields of parabolic geometries},
   journal={Indiana Univ. Math. J.},
   status={to appear},
   eprint={arXiv:1208.5510},
   pages={48 pp.},
}
   
\bib{book}{book}{
   author={{\v{C}}ap, Andreas},
   author={Slov{\'a}k, Jan},
   title={Parabolic geometries. I. Background and general theory},
   series={Mathematical Surveys and Monographs},
   volume={154},
   publisher={American Mathematical Society},
   place={Providence, RI},
   date={2009},
   pages={x+628},
   isbn={978-0-8218-2681-2},
   review={\MR{2532439 (2010j:53037)}},
}

\bib{Frances}{article}{
   author={Frances, Charles},
   title={Local dynamics of conformal vector fields},
   journal={Geom. Dedicata},
   volume={158},
   date={2012},
   pages={35--59},
   issn={0046-5755},
   review={\MR{2922702}},
   doi={10.1007/s10711-011-9619-7},
}

\bib{Frances-Melnick}{article}{
   author={Frances, Charles},
   author={Melnick, Karin},
   title={Formes normales pour les champs conformes
     pseudo-riemanniens.}, 
   journal={Bull. SMF}, 
   status={to appear},
   eprint={arXiv:1008.3781},
   pages={49 pp.},
}

\bib{Nagano-Ochiai}{article}{
   author={Nagano, Tadashi},
   author={Ochiai, Takushiro},
   title={On compact Riemannian manifolds admitting essential projective
   transformations},
   journal={J. Fac. Sci. Univ. Tokyo Sect. IA Math.},
   volume={33},
   date={1986},
   number={2},
   pages={233--246},
   issn={0040-8980},
   review={\MR{866391 (88e:53060)}},
}

\end{biblist}
\end{bibdiv}
\end{document}